\documentclass[sn-mathphys-num]{sn-jnl}

\usepackage{graphicx}%
\usepackage{multirow}%
\usepackage{amsmath,amssymb,amsfonts}%
\usepackage{amsthm}%
\usepackage{mathrsfs}
\usepackage{xcolor}%
\usepackage{textcomp}%
\usepackage{manyfoot}%
\usepackage{booktabs}%
\usepackage{algorithm}%
\usepackage{algorithmicx}%
\usepackage{algpseudocode}%
\usepackage{listings}%

\newtheorem{theorem}{Theorem}

\newtheorem{remark}{Remark}
\usepackage{natbib}
\usepackage{tikz}

\theoremstyle{thmstyleone}%
\theoremstyle{thmstyletwo}%

\theoremstyle{thmstylethree}%

\begin{document}

\title[Article Title]{Asymptotic analysis  of a stochastic SVEIS epidemic model using  Black-Karasinski process}


\author{\fnm{ Lahcen Khammich}}

\author{\fnm{Driss Kiouach}}

\affil{{L2MASI Laboratory},
            {Faculty of sciences Dhar El Mahraz}, 
           {Sidi Mohamed Ben Abdellah University},
            {Fez},
            {Morocco}}


\abstract{In this paper, we present a stochastic SVEIS epidemic model perturbed by a 
Black–Karasinski process. Using a Lyapunov functional approach, we derive a sufficient condition, $\mathcal{R}_0^s > 1$ for the existence of a stationary distribution, which indicates disease persistence. Additionally, we theoretically demonstrate that the disease will die out at an exponential rate if $\mathcal{R}_0^e <1$. Our results show that random fluctuations will facilitate disease outbreak.}




\maketitle
\section{Introduction}
Epidemiological modeling is a crucial tool for understanding and predicting the dynamics of infectious diseases. Since the foundational work of Kermack and McKendrick in the early 20th century \citep{1927}, mathematical models have become central to the study of epidemiology. These models are often based on compartmental frameworks, where the population is divided into distinct groups according to disease status, such as susceptible, exposed, or infected. The evolution of the disease is then described through a system of ordinary differential equations that govern the transitions between these compartments.\\
Mathematical models play a vital role in predicting the progression of an epidemic and designing strategies to mitigate its spread. In this study, we focus on the SVEIS (Susceptible-Vaccinated-Exposed-Infectious-Susceptible) model introduced by Yunquan Song and Xinhong Zhang in \citep{Song2022}. The deterministic form of the SVEIS model is expressed as the following system of equations:

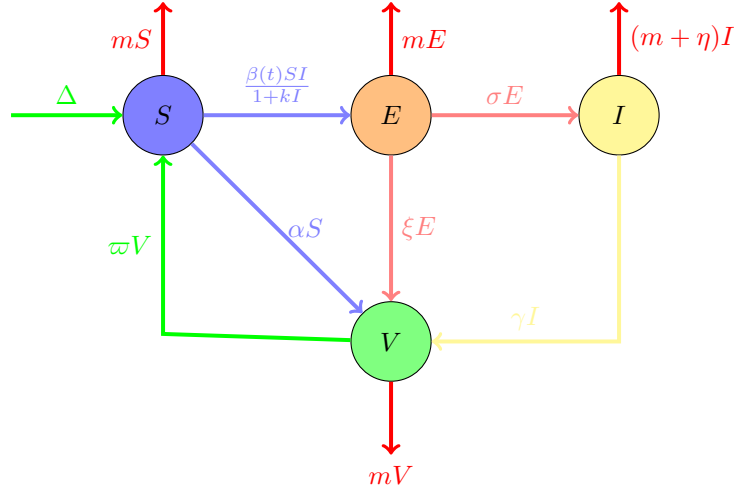
\begin{figure}[H]
\centering
\begin{tikzpicture}[node distance=2cm, auto]
    \tikzstyle{human}=[circle, draw, fill=blue!50, node distance=3cm, minimum height=3em]
    \tikzstyle{infected}=[circle, draw, fill=orange!50, node distance=3cm, minimum height=3em]
    \tikzstyle{vector}=[circle, draw, fill=yellow!50, node distance=3cm, minimum height=3em]
    \tikzstyle{recovered}=[circle, draw, fill=green!50, node distance=3cm, minimum height=3em]
    
    \node[human] (S) {$S$};
    \node[infected, right of=S] (E) {$E$};
    \node[vector, right of=E] (I) {$I$};
    \node[recovered, below of=E] (V) {$V$};
    \draw[->,line width=1.5pt,color=blue!50] (S) -- (E) node[midway, above] 
    {$\frac{\beta(t) SI}{1+kI}$};
    \draw[->,line width=1.5pt,color=green] (-2,0) -- (S) node[midway, above] {$\Delta$};
   \draw[->,line width=1.5pt,color=red] (S) -- (0,1.5) node[midway, left] {$m S$};
   \draw[->,line width=1.5pt,color=green](V) -- (0,-2.9) -- (S) node[midway, left]
    {$\varpi V$};
   \draw[->,line width=1.5pt,color=blue!50] (S) -- (V) node[midway, right] 
    {$\alpha S$};
    \draw[->,line width=1.5pt,color=red!50] (E)  -- (V) node[midway, right] 
    {$\xi E$};
    \draw[->,line width=1.5pt,color=red!50] (E) -- (I) node[midway, above] 
    {$\sigma E$};
    \draw[->,line width=1.5pt,color=red] (E) -- (3,1.5) node[midway, right] {$m E$};
    \draw[->,line width=1.5pt,color=red] (I) -- (6,1.5) node[midway, right] {$(m+\eta) I$}; 
    \draw[->,line width=1.5pt,color=yellow!50] (I) -- (6,-3) -- (V) node[midway, above] {$\gamma I$}; 
    \draw[->,line width=1.5pt,color=red] (V) -- ++(0,-1.5) node[below] {$m V$};
    \end{tikzpicture}
    \caption{Schematic diagram for SEIV model}
    \label{fig:exemple}  
\end{figure}
\begin{equation}
\left\{\begin{array}{l}
\mathrm{d} S(t)=\left(\Pi-\alpha S-\frac{\beta S I}{1+k I}-m S+\omega V\right) \mathrm{d} t, \\
\mathrm{~d} V(t)=(\alpha S+\gamma I+\xi E-(m+\omega) V) \mathrm{d} t, \\
\mathrm{~d} E(t)=\left(\frac{\beta S I}{1+k I}-(\mu+\sigma+\xi) E\right) \mathrm{d} t, \\
\mathrm{~d} I(t)=(\sigma E-(m+\gamma+\eta) I) \mathrm{d} t, \\
\end{array}\right.
\end{equation}\label{1}
\begin{table}[h]
\centering
\begin{tabular}{clccc }
\hline 
Parameters& Description  \\ 
\hline 
$\Pi$&$\text { The Recruitment rate  }  $ \\ 

$ \alpha$& The vaccination rate  \\ 

$\mu$&The natural death rate   \\ 

$\frac{1}{\omega}$&  The average time of immunity waning \\ 
 
$\frac{1}{\gamma}$&  The mean infectious period  \\ 

$ \beta$&  The disease transmission coefficient \\ 

$\frac{1}{\sigma}$&  The latent period   \\ 

$k$ & The inhibition effect  \\ 

$\xi$ &  The recovery rate of exposed class\\  

$\eta$& The rate of disease-related death \\ 
\hline 
\end{tabular}
\caption{List of parameters} 
\end{table}\label{Table1}

Where All parameters are assumed to be positive constants, with their respective descriptions provided in Table \ref{Table1}. By the analysis in \citep{Song2022} the reproduction number is\\ $$\mathcal{R}_0=\frac{\sigma\beta \Pi\left(m+\omega\right)}{m\left(m+\alpha+\omega\right)\left(m+\gamma+\eta\right)\left(m+\sigma+\xi\right)}$$ for more detail about asymptotically analysis and disease equilibrium see \citep{Song2022}.\\
However, due to environmental variability and irregular human activities, disease transmission is often influenced by different types of random noise. As a result, key parameters in the progression of an infectious disease pandemic are not constant but subject to stochastic fluctuations \citep{zhou2023threshold}. This highlights the necessity of studying randomly perturbed SEIS models. Typically, two common approaches are used to incorporate stochastic perturbations: linear Gaussian noise \citep{Lan2021,Zhai2023} and the Ornstein-Uhlenbeck process \citep{Cai2018,Zhou2022}. Notably, Allen \citep{Allen2016} emphasized that the Ornstein-Uhlenbeck process, compared to linear Gaussian noise, has several advantages, including continuity, asymptotic distribution characteristics, and its practical ability to describe the impact of environmental changes on disease dynamics. Based on this reasoning, the incidence rate 
, a critical parameter in the model, can be treated as a stochastic process 
 governed by an Ornstein-Uhlenbeck dynamic., the standard Ornstein-Uhlenbeck process does not guarantee the non-negativity required for epidemiological parameters. To address this, 
 can be modeled using the Black-Karasinski process \citep{Han2023,Han2025}, where the logarithm of 
 satisfies an Ornstein-Uhlenbeck process:
the stochastic model is given by the following equation:
$$
\mathrm{d}\ln \beta=\theta(\ln \bar{\beta}-\ln \beta) d t+\delta d B(t) .
$$
Here, $\bar{\beta}$ denotes the long-term average infection rate, while $B(t)$ represents a standard Brownian motion. Additionally, $\theta$ and $\delta$ are positive constants. In this context, $k$ denotes the speed of reversion, while $\sigma$ indicates the volatility intensity. \\
then by this denotation a stochastic version of \ref{1} is  given by:\\
\begin{equation}
\left\{\begin{array}{l}
 \mathrm{d}\ln \beta=\theta(\ln \bar{\beta}-\ln \beta) d t+\delta d B(t)\\
\mathrm{d} S(t)=\left(\Pi-\alpha S-\frac{\beta S I}{1+k I}-m S+\omega V\right) \mathrm{d} t, \\
\mathrm{~d} V(t)=(\alpha S+\gamma I+\xi E-(m+\omega) V) \mathrm{d} t, \\
\mathrm{~d} E(t)=\left(\frac{\beta S I}{1+k I}-(m+\sigma+\xi) E\right) \mathrm{d} t, \\
\mathrm{~d} I(t)=(\sigma E-(m+\gamma+\eta) I) \mathrm{d} t, \\
\end{array}\right.
\end{equation}\label{2}
 In fact, our main aim is to study the asymptotic behavior of infectious disease, i.e., the long-term properties of $ (S(t),V(t), E(t),I(t))$. By
 letting $z=\ln \beta-\ln \bar{\beta}$ we can equivalently transform system \ref{2}  into \\
\begin{equation}
\left\{\begin{array}{l}
\mathrm{d} S(t)=\left(\Pi-\alpha S-\frac{\bar{\beta} \mathrm{e}^z S I}{1+k I}-\mu S+\omega V\right) \mathrm{d} t, \\
\mathrm{~d} V(t)=(\alpha S+\gamma I+\xi E-(m+\omega) V) \mathrm{d} t, \\
\mathrm{~d} E(t)=\left(\frac{\bar{\beta} \mathrm{e}^z S I}{1+k I}-(m+\sigma+\xi) E\right) \mathrm{d} t, \\
\mathrm{~d} I(t)=(\sigma E-(m+\gamma+\eta) I) \mathrm{d} t, \\
\mathrm{~d} z(t)=-\theta z(t) d t+\delta d B(t)
\end{array}\right.\end{equation}\label{3}\\
The rest of this paper is structured as follows: Section \ref{section2} show the existence and uniqueness of the global solution.  Sections \ref{section3} and \ref{section 4}, present the necessary conditions for the existence of a stationary distribution and the conditions for extinction, respectively. Finally, the main conclusion of the paper are discussed in Section \ref{section 5}.\\
\section{Existence and uniqueness of the global solution}\label{section2}
\begin{theorem}
For any initial condition $(S(0), V(0), E(0), I(0),z(0)) \in \mathbb{R}_+^4 \times \mathbb{R}$ the system \ref{3} admits a unique global solution 
almost surely , and  the solution 
 remains forever in the invariant set:$$\Gamma = \left\{ (S, V, E, I,z) \in \mathbb{R}_+^4 \times \mathbb{R} \, \middle| \, S(t) + V(t) + E(t)+I(t) \leq \frac{\Pi}{m},S\leq S^0 \right\}$$.
\end{theorem}
\begin{remark}
By defining  a desirable non-negative $C^2$-function\\ $W=[S-1-\ln S]+[V-1-\ln V]+[E-1-\ln E]+[I-1-\ln I]+\mathrm{e}^z-1-z$  The remainder of the
 proof is almost the same as Theorem 3.1 in \citep{Han2023} and is thus omitted.
\end{remark}

\section{Stationary distribution }\label{section3}

\begin{theorem}
If $\mathcal{R}_0^s > 1$, where $$\mathcal{R}_0^s=\frac{\bar{\beta}\Pi\sigma(m+\omega)\mathrm{e}^{\frac{\delta^2}{16\theta}}}{m(m+\alpha+\omega)(m+\gamma+\eta)(m+\sigma+\xi)}$$ then the stochastic system described by equation \ref{3} admits at least one ergodic stationary distribution, denoted by $\eta(.)$, within the domain $\Gamma$.
\end{theorem}
\label{theorem}
\begin{proof}
Let defining a series of suitables $C^2$ functions:
$$
\begin{aligned}
&\phi_1= -\ln S -c_1\ln V\\
&\phi_2=-\ln E + c_2\phi_1-c_3\ln I+c_4I\\
&\phi_3= \phi_2 +\frac{c_2 \bar{\beta}}{4c_0(m+\gamma+\eta)}I\\
&\phi_4=-\ln S-\ln V-\ln I-\ln( S^0-S)-\ln(\frac{\Pi}{m}-S-E-V-I)+\mathrm{e}^z-z-1\\
&\phi=M\phi_3+\phi_4\\
\end{aligned}
$$
Applying Ito's formula for the functions above, we have\\
\begin{equation}
\begin{aligned}
\mathscr{L}\phi_1&=-\frac{1}{S}\left(\Pi-\alpha S-\frac{\bar{\beta} \mathrm{e}^z S I}{1+k I}-m S+\omega V\right)-\frac{c_1}{V}\left(\alpha S+\gamma I+\xi E-(m+\omega) V\right)\\
&\leq-\frac{\Pi}{S}+\alpha+m+(\chi -2\sqrt{c_1 \varpi \alpha} +c_1(m+\omega)+\bar{\beta}\mathrm{e}^z I \\
&\leq -\frac{\Pi}{S}+\frac{m(m+\alpha+\omega)}{m+\omega}+\bar{\beta}I(c_0 \mathrm{e}^{2z}+\frac{1}{4 c_0})\\
&\leq  -\frac{\Pi}{S}+\frac{m(m+\alpha+\omega)}{m+\omega}+ \bar{\beta}\frac{\Pi}{m}c_0\left(\mathrm{e}^{2z}-\mathrm{e}^{\frac{\delta^2}{\theta}}\right)+\bar{\beta}c_0 \mathrm{e}^{\frac{\delta^2}{\theta}} \frac{\Pi}{m}+ \frac{\bar{\beta I}}{4c_0}
\end{aligned}
\end{equation}
For $\phi_2$ we have: \\
\begin{equation}
\begin{aligned}
\mathscr{L}\phi_2&\leq -\frac{\bar{\beta} \mathrm{e}^z S I}{E (1+kI)}+(m+\sigma+\xi) -c_2\frac{\Pi}{S}+c_2 \frac{m(m+\alpha+\omega)}{m+\omega}+\bar{\beta}c_0 c_2 \mathrm{e}^{\frac{\delta^2}{\theta}} \frac{\Pi}{m}+\frac{c_2\bar{\beta I}}{4c_0}\\
&-c_3 \frac{\sigma E}{I} +c_3 (m+\gamma+\eta) +c_4 \sigma E -c_4\frac{m+\gamma+\eta}{k}(1+kI)+c_4\frac{m+\gamma+\eta}{k}\\
&+\bar{\beta}\frac{\Pi}{m}c_2 c_0\left(\mathrm{e}^{2z}-\mathrm{e}^{\frac{\delta^2}{\theta}}\right)\\
 &\leq (m+\sigma+\xi) -4 \sqrt[4]{\bar{\beta}\mathrm{e}^z c_2 c_3 c_4 \Pi \sigma \frac{m+\gamma+\eta}{k}}+c_2 \frac{m(m+\alpha+\omega)}{m+\omega}+c_3(m+\gamma+\eta)\\
 &+c_4 \frac{m+\gamma+\eta}{k}+c_4\sigma E +\bar{\beta}c_0 c_2 \mathrm{e}^{\frac{\delta^2}{\theta}} \frac{\Pi}{m}+\frac{c_2\bar{\beta I}}{4c_0}+\bar{\beta}\frac{\Pi}{m}c_2 c_0\left(\mathrm{e}^{2z}-\mathrm{e}^{\frac{\delta^2}{\theta}}\right)\\
 &= (m+\sigma+\xi) -4 \sqrt[4]{\bar{\beta} c_2 c_3 c_4 \Pi \sigma \frac{m+\gamma+\eta }{k} \mathrm{e}^{\frac{\delta^2}{16\theta}}}+c_2 \frac{m(m+\alpha+\omega)}{m+\omega}+c_3(m+\gamma+\eta)\\ 
 &+c_4 \frac{m+\gamma+\eta}{k} +\bar{\beta}c_0 c_2 \mathrm{e}^{\frac{\delta^2}{\theta}} \frac{\Pi}{m}+\frac{c_2\bar{\beta I}}{4c_0}+c_4 \sigma E+\bar{\beta}\frac{\Pi}{m}c_2 c_0\left(\mathrm{e}^{2z}-\mathrm{e}^{\frac{\delta^2}{\theta}}\right)\\&+(4 \sqrt[4]{\bar{\beta} c_2 c_3 c_4 \Pi \sigma \frac{m+\gamma+\eta }{k} })(\mathrm{e}^{\frac{\delta^2}{64\theta}}-\mathrm{e}^{\frac{z}{4}})
\end{aligned}
\end{equation}
For $\phi_3$ we have: \\
\begin{equation}
\begin{aligned}
\mathscr{L}\phi_3&\leq (m+\sigma+\xi) -4 \sqrt[4]{\bar{\beta}\mathrm{e}^z c_2 c_3 c_4 \Pi \sigma \frac{m+\gamma+\eta}{k}}+c_2 \frac{m(m+\alpha+\omega)}{m+\omega}+c_3(m+\gamma+\eta)\\
 &+c_4 \frac{m+\gamma+\eta}{k} +\bar{\beta}c_0 c_2 \mathrm{e}^{\frac{\delta^2}{\theta}} \frac{\Pi}{m}+(c_4+\frac{c_2\bar{\beta}}{4c_0(m+\gamma+\eta)})\sigma E+h_1(z)+h_2(z)\\
 &\leq (m+\sigma+\xi)-\frac{\bar{\beta}\Pi\sigma(m+\omega)\mathrm{e}^{\frac{\delta^2}{16\theta}}}{m(m+\alpha+\omega)(m+\gamma+\eta)}+\bar{\beta}c_0 c_2 \mathrm{e}^{\frac{\delta^2}{\theta}} \frac{\Pi}{m}+(c_4+\frac{c_2\bar{\beta}}{4c_0(m+\gamma+\eta)})\sigma E\\
   &\hspace{1.5cm}+ h_1(z)+h_2(z)
 \end{aligned}
\end{equation}
where we have:\\
$$c_2\frac{m(m+\alpha+\omega)}{m+\omega}=c_3(m+\gamma+\eta)=c_4\frac{m+\gamma+\eta}{k}=\frac{\bar{\beta}\Pi\sigma(m+\omega)\mathrm{e}^{\frac{\delta^2}{16\theta}}}{m(m+\alpha+\omega)(m+\gamma+\eta)}$$
And \\
$$h_1(z)=\bar{\beta}\frac{\Pi}{m}c_2 c_0\left(\mathrm{e}^{2z}-\mathrm{e}^{\frac{\delta^2}{\theta}}\right), h_2=(4 \sqrt[4]{\bar{\beta} c_2 c_3 c_4 \Pi \sigma \frac{m+\gamma+\eta }{k} })(\mathrm{e}^{\frac{\delta^2}{64\theta}}-\mathrm{e}^{\frac{z}{4}})$$
Then we have:\\
\begin{equation}
\begin{aligned}
\mathscr{L}\phi_3&\leq -\frac{(m+\sigma+\xi)(\mathcal{R}_0^s-1)}{2}+(c_4+\frac{c_2\bar{\beta}}{4c_0(m+\gamma+\eta)})\sigma E+ h_1(z)+h_2(z)
\end{aligned}
\end{equation}\label{4}
where \\
$$\mathcal{R}_0^s=\frac{\bar{\beta}\Pi\sigma(m+\omega)\mathrm{e}^{\frac{\delta^2}{16\theta}}}{m(m+\alpha+\omega)(m+\gamma+\eta)(m+\sigma+\xi)},   c_0=\frac{m(m+\sigma+\xi)(\mathcal{R}_0^s-1)}{2\bar{\beta}c_2\mathrm{e}^{\frac{\delta^2}{\theta}}\Pi}$$
For $\phi_4$ and acording to we have: \\
\begin{equation}
\begin{aligned}
\mathscr{L}\phi_4&\leq -\frac{\Pi}{S}+\alpha+m+\bar{\beta}\mathrm{e}^z \frac{\Pi}{m}-\frac{\alpha S}{V} +(m+\omega)-\frac{\sigma E}{I}+(m+\gamma+\eta)+m\\&-\frac{\eta I}{\frac{\Pi}{m}-(S+V+E+I)}+u+\alpha+\omega-\frac{\omega E}{S^0-S}-\theta z(\mathrm{e}^z-1)+\frac{\delta^2}{2}\mathrm{e}^z\\
\mathscr{L}\phi_4&\leq -\frac{\Pi}{S}-\frac{\alpha S}{V}-\frac{\sigma E}{I}-\frac{\eta I}{\frac{\Pi}{m}-(S+V+E+I)}-\frac{\omega E}{S^0-S}\\&\hspace{1.5cm}+5m+2\alpha+2\omega+\gamma+\eta + \left(\bar{\beta}\frac{\Pi}{m}+\frac{\delta^2}{2}\right)\mathrm{e}^z-\theta z(\mathrm{e}^z-1)
\end{aligned}
\end{equation}\label{5}
using the fact $$B:=\sup _{z \in \mathbb{R}}\left\{5m+2\alpha+2\omega+\gamma+\eta + \left(\bar{\beta}\frac{\Pi}{m}+\frac{\delta}{2}\right)\mathrm{e}^z-\frac{\theta z}{2}(\mathrm{e}^z-1)\right\}<\infty$$\\ we can choose a constant $M$ such that
\begin{equation}
\begin{aligned}
-M\frac{(m+\sigma+\xi)(\mathcal{R}_0^s-1)}{2}+B \leq-2 .
\end{aligned}
\end{equation}\label{6}

By combining we obtain for $\phi$ :

$$
\begin{aligned}
\mathscr{L} \phi & \leq-2+M(c_4+\frac{c_2\bar{\beta}}{4c_0(m+\gamma+\eta)})\sigma E -\frac{\Pi}{S}-\frac{\alpha S}{V}-\frac{\sigma E}{I}-\frac{\eta I}{\frac{\Pi}{m}-(S+V+E+I)}-\frac{\omega E}{S^0-S}\\\
&\hspace{1.5cm}-\frac{\theta z}{2}(\mathrm{e}^z-1)+ M h_1(z)+M h_2(z)\\
& :=F(S,V,E,I,z)+ M h_1(z)+M h_2(z) .
\end{aligned}
$$

Now let construct a compact set\\
$$
\begin{aligned}
\mathbb{H}=& \left\{\left(S, V, E, I,z \right) \in \Gamma \mid \epsilon \leq S\leq S^0-\epsilon^2, \epsilon \leq E, \epsilon^2 \leq V,\epsilon^2\leq I\right.\\
&\left. S+E+V+I\leq\frac{\Delta}{m}-\epsilon^3,\frac{-1}{\epsilon}\leq z\leq\frac{1}{\epsilon}\right\}
\end{aligned}
$$
where $\epsilon$ is a small constent satisfying the inequlities bleow.\\
\\let $\Gamma \backslash \mathbb{H}=\mathbb{H}^c=\bigcup_{i=1}^8 \mathbb{H}_i^c$, we have
\begin{equation}
\begin{aligned}
&\begin{aligned}
& \mathbb{H}_1^c =\left\{\left(S, V, E,I, z \right) \in \Gamma \mid E \in(0, \epsilon)\right\},\mathbb{H}_2^c  =\left\{\left(S, V, E,I, z \right) \in \Gamma \mid S \in(0, \epsilon)\right\},\\
& \mathbb{H}_3^c =\left\{\left(S, V, E,I, z \right) \in \Gamma \mid E \in[\epsilon,\infty),I\in(0,\epsilon^2)\right\},  \mathbb{H}_4^c  =\left\{\left(S, V, E,I, \chi \right) \in \Gamma \mid S \in[\epsilon,\infty),V\in(0,\epsilon^2)\right\}, \\&\mathbb{H}_4^c  =\left\{\left(S, V, E,I, z \right) \in \Gamma \mid S \in[\epsilon,\infty),V\in(0,\epsilon^2)\right\}, \\&\mathbb{H}_5^c =\left\{\left(S, V, E,I, z \right) \in \Gamma \mid E \in[\epsilon,\infty),S\in(S^0-\epsilon^2,\infty)\right\}
\\&\mathbb{H}_6^c =\left\{\left(S, V, E,I, z \right) \in \Gamma \mid S+V+E+I\in[\frac{\Delta}{m}-\epsilon^3,\infty),I \in (\epsilon^2,\infty) \right\},\\&\mathbb{H}_8^c  =\left\{\left(S, V, E,I, z \right) \in \Gamma\mid\lvert z\rvert \in(\frac{1}{\epsilon},\infty) \right\}.
\end{aligned}\\
&\text { In view of } \inf _{|z|>\epsilon^{-1}}\left\{z\left(e^z-1\right)\right\}=\frac{1}{\epsilon}\left(1-e^{-\frac{1}{\epsilon}}\right) \text {, we obtain }\\
&F(S,V,E,I,z) \leq \begin{cases}-2+M(c_4+\frac{c_2\bar{\beta}}{4c_0(m+\gamma+\eta)})\sigma\epsilon \leq-1, & \text { if } (S,V,E,I,z) \in \mathbb{H}_1^c, \\ -2+M(c_4+\frac{c_2\bar{\beta}}{4c_0(m+\gamma+\eta)})\sigma\frac{\Pi}{m}-\frac{\min \{\Pi,\alpha,\sigma,\eta,\omega \}}{\epsilon} \leq-1, & \text { if } (S,V,E,I,z) \in \bigcup_{i=2}^6 \mathbb{H}_i^c, \\ -2+M(c_4+\frac{c_2\bar{\beta}}{4c_0(m+\gamma+\eta)})\sigma\frac{\Pi}{m}-\frac{\theta}{2 \epsilon}\left(1-e^{-\frac{1}{\epsilon}}\right) \leq-1, & \text { if } (S,V,E,I,z) \in \mathbb{H}_7^c .\end{cases}
\end{aligned}
\end{equation}
In summury we have $F(S, V, E, I, z ) \leq-1$ for any $\left(S, V, E,I, z\right) \in \Gamma \backslash \mathbb{H}:=\mathbb{H}^c$.\\
Following a standard argument presented in \citep{Han2023}, we find that the process $\{z(t)\}_{t \geq 0}$ converges weakly to the unique stationary distribution, which is the normal distribution $\mathbb{N}\left(0, \frac{\delta^2}{2 \theta}\right)$. This stationary distribution has a density function given by :
$\pi(z)=\frac{\sqrt{\theta}}{\delta \sqrt{\pi}} e^{-\frac{\theta}{\delta^2} z^2}, \forall z \in \mathbb{R}$,  and we have for any $a>0$:\\
$$\int_{\mathbb{R}} e^{a x} \pi(z) d z=\frac{\sqrt{\theta}}{\delta \sqrt{\pi}} \int_{\mathbb{R}} e^{-\left(\frac{\sqrt{\theta} z}{\delta}-\frac{a \delta}{2 \sqrt{\theta}}\right)^2+\frac{(a \sigma)^2}{4 \theta}} d z=e^{\frac{(a \sigma)^2}{4 \theta}}, \quad a.s.$$ So and by using  the Itô’s integral, we have:\\
\begin{equation}
\begin{aligned}
0 \leq \frac{\mathbb{E} \phi(S,V,E,I,z)}{T}=&\frac{\mathbb{E} \phi(S(0),V(0),E(0),I(0),z(0))}{T}\\&+\frac{1}{T} \int_0^T \mathbb{E}(F(S(t),V(t),E(t),I(t),z(t))) d t\\&+M \mathbb{E}\left(\frac{1}{T} \int_0^T h_1(z(t)) d t\right) +M \mathbb{E}\left(\frac{1}{T} \int_0^T h_2(z(t)) d t\right).
\end{aligned}
\end{equation}
Using the fact that\\ $$F(S(t),V(t),E(t),I(t),z(t))\leq M(c_4+\frac{c_2\bar{\beta}}{4c_0(m+\gamma+\eta)})\sigma\frac{\Pi}{m}:=K \hspace{0.6cm} \forall \left(S, V, E, I,z \right) \in \Gamma$$
\begin{equation}
\begin{aligned}
\frac{1}{T} \int_0^T \mathbb{E}(F(S(t),V(t),E(t),I(t),z(t))) d t & \leq \frac{K}{T} \int_0^T \mathbf{1}_{\left\{S(t),V(t),E(t),I(t),z(t) \in \mathbb{H}\right\}} d t\\
&-\frac{1}{T} \int_0^T \mathbf{1}_{\left\{S(t),V(t),E(t),I(t),z(t) \in \Gamma \backslash \mathbb{H}\right\}} d t \\
& =-1+\frac{K+1}{T} \int_0^T \mathbf{1}_{\left\{S(t),V(t),E(t),I(t),z(t) \in \mathbb{H}\right\}} d t .
\end{aligned}
\end{equation}
In addition and by the ergodicité of   $\{z(t)\}_{t \geq 0}$ we have\\
$\lim _{T \rightarrow \infty}\mathbb{E}\left(\frac{1}{T} \int_0^T h_1(z(t)) d t\right)=0, $ a.s and $\lim _{T \rightarrow \infty}\mathbb{E}\left(\frac{1}{T} \int_0^T h_2(z(t)) d t\right)=0,$a.s\\
By combining we get 
$$
\liminf _{T \rightarrow \infty} \frac{1}{T} \int_0^T \mathbf{1}_{\left\{S(t),V(t),E(t),I(t),z(t) \in \mathbb{H}\right\}} d t \geq \frac{1}{K+1}>0, \quad \text { a.s. }
$$

Then by Fatou's lemma, it implies

$$
\liminf _{T \rightarrow \infty} \frac{1}{T} \int_0^T \mathbb{P}\left(S(0),V(0),E(0),I(0),z(0), \mathbb{H}, t\right) d t \geq \frac{1}{K+1}, \quad \text { for all } \left(S(0),V(0),E(0),I(0),z(0)\right) \in \Gamma,
$$

where $\mathbb{P}\left(S(0),V(0),E(0),I(0),z(0), \mathbb{H}, t\right)$ denotes the transition probability of \\$(S(t),V(t),E(t),I(t),z(t)) \in \mathbb{H}$ with initial value $(S(0),V(0),E(0),I(0),z(0))$.\\ which complete the proof.

\end{proof}
\section{Extinction }\label{section 4}
\begin{theorem}
For any initial value $\left(S(0), V(0), E(0), I(0), \chi (0)\right)\in\Gamma$, the solution\\ $\left(S(t), V(t), E(t), I(t), \chi (t)\right)$ of system \ref{2} has the following property\\
$$\mathop {\lim \sup }\limits_{t \to  + \infty } \frac{{\ln \left( {\frac{{{\omega _1}}}{{{m} + \sigma+\xi }}{E} + \frac{{{\omega _2}}}{{{m} + \gamma+\eta }}{I}} \right)}}{t} \le \min \left\{ {{m} + \sigma+\xi ,{m} + \gamma+\eta } \right\}\left( {\mathcal{R}_0^e - 1} \right)$$a.s.,\\ Where $\omega_1=\frac{\sigma}{\left(m+\gamma+\eta\right)\sqrt{\mathcal{R}_0}}$, $\omega_2=1$, ${\mathcal{R}_0^e}=\sqrt {\mathcal{R}_0}  + \frac{\sigma S^0\bar{\beta} \sqrt{e^{\frac{\delta^2}{\theta}}-2 e^{\frac{\delta^2}{4 \theta}}+1} }{\sqrt{\mathcal{R}_0} \left( m + \sigma+\xi \right)\min \left( {m+\sigma+\xi , m + \gamma+\eta } \right)}$\\
\\
Especially, if   ${\mathcal{R}_0^e}<1$, $\mathop {\lim }\limits_{t \to  + \infty }E\left( t \right) = 0$, $\mathop {\lim }\limits_{t \to  + \infty }I\left( t \right) = 0$,  a.s.
\end{theorem}
\begin{proof}
Firstly, we define a matrix\\ 
$$M = \left( {\begin{array}{*{20}{c}}
0&{\frac{{\bar{\beta } S^0}}{{m+\sigma+\xi}}}\\
{\frac{\sigma }{{m+\gamma+\eta }}}&0
\end{array}} \right)$$\\
 As stated in \citep{berman1994nonnegative}, it is guaranteed that for a non-singular matrix $M$ , a left eigenvector $(\omega_1,\omega_2)$ exists for its eigenvalue  $\sqrt{\mathcal{R}_0}$, such that $\sqrt{\mathfrak{R}_0}\left(\omega_1,\omega_2\right)=\left(\omega_1,\omega_2\right)M$
\\ where $\omega_1=\frac{\sigma}{(m+\gamma+\eta)\sqrt{\mathcal{R}_0}}$, $\omega_2=1$.\\  Denoting\\
\\
$V_e=\frac{\omega_1}{m+\sigma+\xi} E+\frac{\omega_2}{m+\gamma+\eta} I$, we have
$$
\begin{aligned}
\mathcal{L}\left(\ln V_e\right) & =\frac{1}{V_e}\left(\frac{\omega_1}{m+\sigma+\xi} \frac{\bar{\beta} e^z S I}{1+k I}-\varpi_1 E+\frac{\omega_2}{m+\gamma+\eta} \sigma E-\omega_2 I\right) \\
& \leq \frac{1}{V_e}\left(\frac{\omega_1}{m+\sigma+\xi} \bar{\beta } S^0 I-\omega_1 E+\frac{\omega_2}{m+\gamma+\eta} \sigma E-\omega_2 I\right)+\frac{1}{V_e} \frac{\omega_1}{m+\sigma+\xi}\left(e^z-1\right) \bar{\beta }  S^0 I \\
& \leq \frac{1}{V_e}\left(\omega_1, \omega_2\right)\left(M-I_2\right)\left[\begin{array}{c}
E \\
I
\end{array}\right]+\frac{I}{V_e} \frac{\omega_1 S^0 \bar{\beta }}{m+\sigma+\xi}|e^z-1| \\
& =\frac{1}{\frac{\omega_1}{m+\sigma+\xi} E+\frac{\omega_2}{m+\gamma+\eta} I}\left(\sqrt{\mathcal{R}_0}-1\right)\left(\omega_1 E+\omega_2 I\right)+\frac{I}{\frac{\omega_1}{m+\sigma+\xi} E+\frac{\omega_2}{m+\gamma+\eta} I} \frac{\omega_1 S^0 \bar{\beta}}{m+\sigma+\xi}|e^z-1| \\
& \leq \min \{m+\sigma+\xi, m+\gamma+\eta\}\left(\sqrt{\mathcal{R}_0}-1\right)+\frac{\sigma S^0 \bar{\beta}}{\sqrt{\mathcal{R}_0}(m+\sigma+\xi)}|e^z-1| .
\end{aligned}
$$
By the ergodicity of $\{z(t)\}_{t \geq 0}$ and the Hölder's inequality, we have :\\
$$
\begin{aligned}
\lim _{t \rightarrow \infty} \frac{1}{t} \int_0^t\left|e^{z(s)}-1\right| d s=\int_{\mathbb{R}}\left|e^z-1\right| \pi(z) d z \leq& \left(\int_{\mathbb{R}}\left(e^z-1\right)^2 \pi(z) d z\right)^{\frac{1}{2}}\\=& \sqrt{e^{\frac{\delta^2}{\theta}}-2 e^{\frac{\delta^2}{4 \theta}}+1}, \quad \text { a.s. }
\end{aligned}
$$

Integrating from $0$ to $t$ and then dividing by $t$ on both sides, and combining, we get \\
$$
\begin{aligned}
\limsup _{t \rightarrow \infty} \frac{\ln \left(\frac{\varpi_1}{m+\sigma+\xi} E+\frac{\omega_2}{m+\gamma+\eta} I\right)}{t} \leq & \min \{m+\sigma+\xi, m+\gamma+\eta\}\left(\sqrt{\mathcal{R}_0}-1\right)\\&+\frac{\sigma S^0 \bar{\beta}}{\sqrt{\mathcal{R}_0}(m+\sigma+\xi)} \sqrt{e^{\frac{\delta^2}{\theta}}-2 e^{\frac{\delta^2}{4 \theta}}+1}\\
= & \min \{m+\sigma+\xi, m+\gamma+\eta\}\left(\mathcal{R}_0^e-1\right) .
\end{aligned}
$$
Therefore, if $\mathcal{R}_0^e<1$, then the disease will go to extinction. this completes the proof.
\end{proof}
\section{Conclusion}\label{section 5}
This work focuses on a stochastic SVEIS epidemic model \ref{3}, where the Black-Karasinski process is introduced to represent the random effects in disease transmission. We define two critical values $\mathcal{R}_0^s$
and $\mathcal{R}_0^e$, to analyze the asymptotic behavior of the system. Specifically, we demonstrate the existence of a stationary distribution when $\mathcal{R}_0^s>1$,and establish the exponential extinction of the disease when $\mathcal{R}_0^e<1$.

\section*{Declarations}
The authors declare that there is no conflict of interest regarding the publication of this paper.
\bibliographystyle{plainnat}
\bibliography{SNbiblio}
\end{document}